\documentclass[12pt,a4paper]{article}
\usepackage{indentfirst}
\usepackage{amsfonts}
\usepackage{amssymb}
\usepackage{mathrsfs}
\usepackage{amsmath}
\usepackage{amsthm}
\usepackage{enumerate}
\usepackage{cite}
\allowdisplaybreaks
\newtheorem{defn}{Definition}[section]
\newtheorem{thm}[defn]{Theorem}
\newtheorem{lem}[defn]{Lemma}

\newtheorem{cor}[defn]{Corollary}
\newtheorem{ex}[defn]{Example}
\newtheorem{re}[defn]{Remark}
\begin{document}
\title{{\bf Derivations on semi-simple Jordan algebras and its applications}}
\author{\normalsize \bf Chenrui Yao,  Yao Ma,  Liangyun Chen}
\date{{\small{ School of Mathematics and Statistics,  Northeast Normal University,\\ Changchun 130024, CHINA
}}} \maketitle
\date{}

   {\bf\begin{center}{Abstract}\end{center}}

In this paper, we mainly study the derivation algebras of semi-simple Jordan algebras over a field of characteristic $0$ and give sufficient and necessary conditions that the derivation algebras of them are simple. As an application, we prove that for a semi-simple Jordan algebra $J$, $TDer(Der(J)) = Der(Der(J)) = ad(Der(J))$ under some assumptions. Moreover, we also show that for a semi-simple Jordan algebra $J$ which has a finite basis over a field of characteristic $0$, $TDer(J) = Der(J) = Inn(J)$. This is a corollary about our theorem which concerns Jordan algebras with unit.

\renewcommand{\thefootnote}{\fnsymbol{footnote}}
\footnote[0]{ Corresponding author(L. Chen): chenly640@nenu.edu.cn.}
\footnote[0]{Supported by  NNSF of China (Nos. 11771069 and 11801066), NSF of Jilin province (No. 20170101048JC), the project of Jilin province department of education (No. JJKH20180005K) and the Fundamental Research Funds for the Central Universities(No. 130014801).}

\section{Introduction}

An algebra $J$ over a field $\rm{F}$ is called a Jordan algebra if
\begin{gather*}
x \circ y = y \circ x,\quad\forall x, y \in J,\\
(x^{2} \circ y) \circ x = x^{2} \circ (x \circ y),\quad x^{2} = x \circ x ,\quad\forall x, y \in J.
\end{gather*}
In 1933, this kind of algebras first appeared in a paper by P. Jordan in his study of quantum mechanics. Subsequently, P. Jordan, J. von Neumann and E. Wigner introduced finite-dimensional formally real Jordan algebras for quantum mechanics formalism in \cite{JNW}. All simple finite-dimensional Jordan algebras over an algebraically closed field $\rm{F}$ of characteristic different from 2 were classified some years later by A. A. Albert in \cite{A1}, using the idempotent method. Since then, Jordan algebras were found various applications in mathematics and theoretical physics (see \cite{1,2,3,4} and references therein) and now form an intrinsic part of modern algebra. Latest results on Jordan algebras refer to \cite{5,6,7} and references therein.

Jordan systems arise naturally as ``coordinates" for Lie algebras having a grading into $3$ parts. Over the years, many predecessors have succeeded in generalizing some of the results of Lie algebras to Jordan algebras. In \cite{A1}, A. A. Albert proved Lie's theorem, Engel's theorem and Cartan's theorem for Jordan algebras. In \cite{N1}, N. Jacobson successfully showed that $Der(J) = Inn(J)$ where $J$ was a semi-simple Jordan algebra over a field of characteristic zero. In \cite{M2}, D. J. Meng proved that if a centerless Lie algebra $L$ had the decomposition $L = L_{1} \oplus L_{2}$, then the derivation algebra of $L$ also had decomposition $Der(L) = Der(L_{1}) \oplus Der(L_{2})$. In this paper, we successfully generalize this result to Jordan algebras(see Theorem \ref{thm:2.2}).

However, there are also some results which couldn't be generalized from Lie algebras to Jordan algebras. For instance, it is well known that a simple Lie algebra $L$ is isomorphic to its derivation algebra $Der(L)$. The above results doesn't hold in the case of Jordan algebras, since the derivation algebras of Jordan algebras are actually Lie algebras rather than Jordan algebras. Moreover, we also know that the derivation algebras of simple Lie algebras are simple. Whether it is also true in simple Jordan algebras or not? The answer is not, which can be deduced by Theorem \ref{thm:2.8}. Since the result is not necessarily true, the condition for a simple Jordan algebra with simple derivation algebra is given in Theorem \ref{thm:2.6}, \ref{thm:2.8} and \ref{thm:2.10}. In this paper, we mainly study simple Jordan algebras over a field whose characteristic is $0$. And we give the sufficient and necessary conditions that the derivation algebras of them are simple.

Derivations have been a historic and widely studied subject for many years. Recall that a derivation on an algebra $A$ is a linear map $D : A \rightarrow A$ satisfying
\[D(xy) = D(x)y + xD(y),\quad\forall x, y \in A.\]
In \cite{N2}, N. Jacobson showed that every nilpotent Lie algebra had a derivation $D$ which was not inner. Moreover, he also proved that if $L$ was a finite dimensional Lie algebra over the field $\rm{F}$ such that the killing form of $L$ was non-degenerate, then the derivations of $L$ were all inner in \cite{N3}. As a generalization, S. Berman proved that if $L$ denoted Lie algebras which generalized the split simple Lie algebras, then the dimension of $Der(L)/ad(L)$ equaled the nullity of the Cartan matrix which defined $L$ in \cite{B1}.

As a nature generalization of derivations, triple derivations appeared in order to study associative algebras(rings). A linear map $D : L \rightarrow L$ where $L$ is a Lie algebra, is called a triple derivation on $L$ if it satisfies
\[D([[x, y], z]) = [[D(x), y], z] + [[x, D(y)], z] + [[x, y], D(z)],\quad\forall x, y, z \in L.\]
Similarly, one can define triple derivations on Jordan algebras. Obviously, derivations are triple derivations. Naturally, one may has such a question that if triple derivations are all derivations. The answer is not trivial. In \cite{Z1}, J. H. Zhou came to a conclusion that if $L$ denoted a centerless perfect Lie algebra over a field whose characteristic was not $2$, then every triple derivation on $L$ was a derivation. What is the answer to this question in Jordan algebras? In this paper, we will study on Jordan algebras with unit and show that triple derivations are all derivations in the case of the characteristic of the basic field is not $2$. As an application, this result is valid on semi-simple Jordan algebras over a field of characteristic $0$ since a semi-simple Jordan algebra over a field of characteristic $0$ has the unique unit.

In the following, we'll give the definition of the center of a Jordan algebra, which is denoted by $C(J)$. The center of a Jordan algebra $J$ is the set
\[\{a \in J | (a \circ x) \circ y = (a \circ y) \circ x = (x \circ y) \circ a, \forall x, y \in J\}.\]

The paper is organised as follows: In Section \ref{se:2}, we'll prove that our main theorem, Theorem \ref{thm:2.2}, which is a generalization of a famous theorem in Lie algebras. Applying this theorem, we can reduce the study of the derivation algebras of semi-simple Jordan algebras to simple Jordan algebras. In the following, we'll study four kinds of simple Jordan algebras over a field of characteristic $0$ respectively and give the sufficient and necessary conditions that the derivation algebras of them are simple, Theorem \ref{thm:2.6} and \ref{thm:2.8}, which can be viewed as the most important results in this paper. In Section \ref{se:3}, we study Jordan algebras with unit over a field of characteristic not $2$ and prove that for such Jordan algebra $J$, $TDer(J) = Der(J)$(see Theorem \ref{thm:3.4}). As an application, we get a corollary that for a semi-simple Jordan algebra $J$ over a field of characteristic $0$, $TDer(J) = Der(J) = Inn(J)$(see Corollary \ref{cor:3.7}). In Section \ref{se:4}, we also show that under some assumptions, $TDer(Der(J)) = Der(Der(J)) = ad(Der(J))$ where $J$ is a semi-simple algebra, i.e., Theorem \ref{thm:4.7}.

\section{Derivation algebras of semi-simple Jordan algebras}\label{se:2}

\begin{lem}\cite{N1}\label{le:2.1}
Let $J$ be a semi-simple Jordan algebra over a field $\rm{F}$. Then $J$ has the decomposition $J = \oplus^{s}_{i = 1}J_{i}$ where $J_{i}(1 \leq i \leq s)$ are simple ideals of $J$.
\end{lem}
\begin{thm}\label{thm:2.2}
Suppose that $J$ is a Jordan algebra and has a decomposition $J = J_{1} \oplus J_{2}$, where $J_{1}$, $J_{2}$ are ideals of $J$. Then
\begin{enumerate}[(1)]
\item $C(J) = C(J_{1}) \dotplus C(J_{2})$;
\item If\; $C(J) = \{0\}$, then $Der(J) = Der(J_{1}) \oplus Der(J_{2})$.
\end{enumerate}
\end{thm}
\begin{proof}
(1). Obviously, $C(J_{1}) \cap C(J_{2}) = \{0\}$. For all $z_{i} \in C(J_{i})(i = 1, 2)$, take $x = x_{1} + x_{2}$, $y = y_{1} + y_{2} \in J$, where $x_{1}, y_{1} \in J_{1}$, $x_{2}, y_{2} \in J_{2}$. We have
\[((z_{1} + z_{2}) \circ x) \circ y = ((z_{1} + z_{2}) \circ (x_{1} + x_{2})) \circ (y_{1} + y_{2}) = (z_{1} \circ x_{1}) \circ y_{1} + (z_{2} \circ x_{2}) \circ y_{2},\]
\[((z_{1} + z_{2}) \circ y) \circ x = ((z_{1} + z_{2}) \circ (y_{1} + y_{2})) \circ (x_{1} + x_{2}) = (z_{1} \circ y_{1}) \circ x_{1} + (z_{2} \circ y_{2}) \circ x_{2},\]
\[(z_{1} + z_{2}) \circ (x \circ y) = (z_{1} + z_{2}) \circ ((x_{1} + x_{2}) \circ (y_{1} + y_{2})) = z_{1} \circ (x_{1} \circ y_{1}) + z_{2} \circ (x_{2} \circ y_{2}),\]
since $z_{i} \in C(J_{i})(i = 1, 2)$, we have
\[(z_{i} \circ x_{i}) \circ y_{i} = (z_{i} \circ y_{i}) \circ x_{i} = z_{i} \circ (x_{i} \circ y_{i})\;(i = 1, 2).\]
Hence,
\[((z_{1} + z_{2}) \circ x) \circ y = ((z_{1} + z_{2}) \circ y) \circ x = (z_{1} + z_{2}) \circ (x \circ y),\]
which implies that $z_{1} + z_{2} \in C(J)$, i.e., $C(J_{1}) \dotplus C(J_{2}) \subseteq C(J)$.

On the other hand, for all $a \in C(J)$, suppose that $a = a_{1} + a_{2}$ where $a_{i} \in J_{i}(i = 1, 2)$. Then for all $x_{1}, y_{1} \in J_{1}$,
\[(a_{1} \circ x_{1}) \circ y_{1} = ((a - a_{2}) \circ x_{1}) \circ y_{1} = (a \circ x_{1}) \circ y_{1},\]
\[(a_{1} \circ y_{1}) \circ x_{1} = ((a - a_{2}) \circ y_{1}) \circ x_{1} = (a \circ y_{1}) \circ x_{1},\]
\[a_{1} \circ (x_{1} \circ y_{1}) = (a - a_{2}) \circ (x_{1} \circ y_{1}) = a \circ (x_{1} \circ y_{1}),\]
since $a \in C(J)$, we have
\[(a \circ x_{1}) \circ y_{1} = (a \circ y_{1}) \circ x_{1} = a \circ (x_{1} \circ y_{1}).\]
Hence,
\[(a_{1} \circ x_{1}) \circ y_{1} = (a_{1} \circ y_{1}) \circ x_{1} = a_{1} \circ (x_{1} \circ y_{1}),\]
which implies that $a_{1} \in C(J_{1})$. Similarly, we have $a_{2} \in C(J_{2})$.

Therefore, we have $C(J) = C(J_{1}) \dotplus C(J_{2})$.

(2). $\bf{Step 1}$. We'll show that $\forall i = 1, 2$, $D(J_{i}) \subseteq J_{i},\quad\forall D \in Der(J)$.

Suppose that $x_{i} \in J_{i}$, then
\[D(x_{1}) \circ x_{2} = D(x_{1} \circ x_{2}) - x_{1} \circ D(x_{2}) \in J_{1} \cap J_{2} = 0,\]
since $J_{1}$, $J_{2}$ are ideals of $J$.
Suppose that $D(x_{1}) = u_{1} + u_{2}$ where $u_{1} \in J_{1}$, $u_{2} \in J_{2}$. Then
\[u_{2} \circ x_{2} = (u_{1} + u_{2}) \circ x_{2} = D(x_{1}) \circ x_{2} = 0,\]
which implies that $u_{2} \in C(J_{2})$. Note that $C(J) = \{0\}$, we have $C(J_{i}) = \{0\}(i = 1, 2)$. Hence, $u_{2} = 0$. That is to say $D(x_{1}) \in J_{1}$. Similarly, we have $D(x_{2}) \in J_{2}$.

$\bf{Step 2}$. We'll show that $Der(J_{1}) \dotplus Der(J_{2}) \subseteq Der(J)$.

For any $D \in Der(J_{1})$, we extend it to a linear map on $J$ as follow
\[D(x_{1} + x_{2}) = D(x_{1}),\quad\forall x_{1} \in J_{1}, x_{2} \in J_{2}.\]

Then for any $x, y \in J$, suppose that $x = x_{1} + x_{2}$, $y = y_{1} + y_{2} \in J$, where $x_{1}, y_{1} \in J_{1}$, $x_{2}, y_{2} \in J_{2}$, we have
\[D(x \circ y) = D((x_{1} + x_{2}) \circ (y_{1} + y_{2})) = D(x_{1} \circ y_{1} + x_{2} \circ y_{2}) = D(x_{1} \circ y_{1}),\]
\[D(x) \circ y + x \circ D(y) = D(x_{1} + x_{2}) \circ (y_{1} + y_{2}) + (x_{1} + x_{2}) \circ D(y_{1} + y_{2}) = D(x_{1}) \circ y_{1} + x_{1} \circ D(y_{1}),\]
since $D \in Der(J_{1})$,
\[D(x_{1} \circ y_{1}) = D(x_{1}) \circ y_{1} + x_{1} \circ D(y_{1}),\]
we have
\[D(x \circ y) = D(x) \circ y + x \circ D(y),\]
which implies that $D \in Der(J)$, i.e., $Der(J_{1}) \subseteq Der(J)$. Moreover, $D \in Der(J_{1})$ if and only if $D(x_{2}) = 0,\; \forall x_{2} \in J_{2}$.

Similarly, we have $Der(J_{2}) \subseteq Der(J)$ and $D \in Der(J_{2})$ if and only if $D(x_{1}) = 0,\; \forall x_{1} \in J_{1}$.

Then we have $Der(J_{1}) + Der(J_{2}) \subseteq Der(J)$ and $Der(J_{1}) \cap Der(J_{2}) = \{0\}$. Hence, $Der(J_{1}) \dotplus Der(J_{2}) \subseteq Der(J)$.

$\bf{Step 3}$. We'll prove that $Der(J_{1}) \dotplus Der(J_{2}) = Der(J)$.

Suppose that $D \in Der(J)$. Set $x = x_{1} + x_{2}, x_{i} \in J_{i}$. Define $D_{1}, D_{2}$ as follows
\begin{equation}
\left\{
\begin{aligned}
D_{1}(x_{1} + x_{2}) = D(x_{1}),\\
D_{2}(x_{1} + x_{2}) = D(x_{2}).
\end{aligned}
\right.
\end{equation}
Obviously, $D = D_{1} + D_{2}$.

For any $u_{1}, v_{1} \in J_{1}$,
\[D_{1}(u_{1} \circ v_{1}) = D(u_{1} \circ v_{1}) = D(u_{1}) \circ v_{1} + u_{1} \circ D(v_{1}) = D_{1}(u_{1}) \circ v_{1} + u_{1} \circ D_{1}(v_{1}).\]
Hence, $D_{1} \in Der(J_{1})$. Similarly, $D_{2} \in Der(J_{2})$.

Therefore, $Der(J) = Der(J_{1}) \dotplus Der(J_{2})$.

$\bf{Step 4}$. We'll show that $Der(J_{i}) \lhd Der(J)$. Suppose that $D_{1} \in Der(J_{1}), D \in Der(J), x_{2} \in J_{2}$,
\[[D, D_{1}](x_{2}) = DD_{1}(x_{2}) - D_{1}D(x_{2}) = 0,\]
which implies that $[D, D_{1}] \in Der(J_{1})$, i.e., $Der(J_{1}) \lhd Der(J)$. Similarly, $Der(J_{2}) \lhd Der(J)$.

Therefore, we have $Der(J) = Der(J_{1}) \oplus Der(J_{2})$.
\end{proof}

According to Lemma \ref{le:2.1} and Theorem \ref{thm:2.2}, we only need to study the derivation algebras of simple Jordan algebras in order to study the derivation algebras of semi-simple Jordan algebras. Moreover, in \cite{FN1}, we get that the simple Jordan algebras over a field whose characteristic is $0$ fall into three ``great" classes and one exceptional class as following:
\begin{enumerate} [(A)]
\item The special Jordan algebras generated by simple associative algebras $\mathscr{A}$ with the multiplication $a \circ b = \frac{1}{2}(ab + ba)$, denoted by $\mathscr{A}^{+}$;
\item The Jordan algebras $H(\mathscr{A}, P)$ of $P$-symmetric elements in simple involutorial algebras $\mathscr{A}$ with an involution $P$;
\item The Jordan algebras constructed by the algebras that define Clifford systems. These have a basis $\{1, u_{1}, \cdots, u_{n}\}$ such that $1$ is the unit and $u^{2}_{i} = \alpha_{i}1$ where $\alpha_{i} \neq 0, u_{i} \circ u_{j} = 0$ if $i \neq j$;
\item The exceptional Jordan algebra corresponding to the system $M^{8}_{3}$ which has $27$ dimensions over its center.
\end{enumerate}

In the following part, we'll study the derivation algebras of the above four kinds of simple Jordan algebras respectively and give the sufficient and necessary conditions that the derivation algebras of them are simple.

\begin{lem}\cite{FN1}\label{le:2.3}
Let $\mathscr{A}$ be a simple associative algebra. Then if $D$ is a derivation on the Jordan algebra $\mathscr{A}^{+}(i.e., Jordan\;algebras\;of\;type\;A)$ there exists an element $d$ in $\mathscr{A}$ such that $D(a) = [a, d]$ for all $a$.
\end{lem}
\begin{lem}\cite{FN1}\label{le:2.4}
Let $\mathscr{A}$ be a simple associative algebra that has an involution $P$ (an anti-isomorphism of $\mathscr{A}$ of period $2$). Then if $D$ is a derivation on $H(\mathscr{A}, P)(i.e., Jordan\;algebras\\\;of\;type\;B)$ there exists a $P$-skew element $d$ in $\mathscr{A}$ such that $D(a) = [a, d]$.
\end{lem}
\begin{lem}
Let $\mathscr{A}$ be a simple associative algebra that has an involution $P$ (an anti-isomorphism of $\mathscr{A}$ of period $2$) and $H = \{a \mid a \in \mathscr{A}, P(a) = -a\}$. Then $H$ is closed under the usual Lie bracket.
\end{lem}
\begin{proof}
For any $x, y \in H$, we have
\begin{align*}
&P([x, y]) = P(xy - yx) = P(xy) - P(yx) = P(y)P(x) - P(x)P(y)\\
&= (-y)(-x) - (-x)(-y) = yx - xy = -[x, y],
\end{align*}
which implies that $[x, y]$ is $P$-skew. Hence, $H$ is closed under the usual Lie bracket.
\end{proof}

We denote such $D$ in Lemma \ref{le:2.3} and Lemma \ref{le:2.4} by $D_{d}$.
\begin{thm}\label{thm:2.6}
Suppose that $J$ is a simple Jordan algebra of type A(respectively, of type B). Let $\mathscr{A}$ be the associated simple associative algebra and $L$ the Lie algebra generated by $\mathscr{A}$(respectively, all $P$-skew elements in $\mathscr{A}$). Then the following are equivalent
\begin{enumerate}[(1)]
\item $Der(J)$ is simple;
\item $L/Z(L)$ is simple where $Z(L)$ denotes the center of $L$.
\end{enumerate}
\end{thm}
\begin{proof}
(1). Suppose that $L/Z(L)$ is simple. We show that $Der(J)$ is simple by using reduction to absurdity.

Otherwise, there exists a non trivial ideal of $Der(J)$, denoted by $D_{S}$. According to Lemma \ref{le:2.3}(respectively, Lemma \ref{le:2.4}), there exists an element(respectively, a $P$-skew element) $d$ in $\mathscr{A}$ such that $D = D_{d}$ for any $D \in Der(J)$. Let $S = \{\bar{d} \mid D_{d} \in D_{S}\}$. It's obvious that $S$ is a subspace of $L/Z(L)$.

Since $D_{S} \neq \{0\}$, there exists a nonzero element $D_{d_{0}} \in D_{S}$. Then $\exists 0 \neq a \in J$ such that
\[[a, d_{0}] = D_{d_{0}}(a) \neq 0,\]
which implies that $d_{0} \notin Z(L)$. Hence, $\bar{d_{0}} \neq \bar{0}$, $S \neq \{\bar{0}\}$.

Since $D_{S} \neq Der(J)$, there exists a nonzero element $D_{d_{1}} \in Der(J)$ and $D_{d_{1}} \notin D_{S}$. That is to say $\bar{0} \neq \bar{d_{1}} \in L/Z(L)$ and $\bar{d_{1}} \notin S$, i.e., $S \neq L/Z(L)$.

For any $a, b \in L, c \in J$, we have
\begin{align*}
&[D_{a}, D_{b}](c) = D_{a}D_{b}(c) - D_{b}D_{a}(c) = [[c, b], a] - [[c, a], b] = [[c, b], a] + [[a, c], b]\\
&= -[[b, a], c] = [c, [b, a]] = -[c, [a, b]] = -D_{[a, b]}(c).
\end{align*}
Hence we have $[D_{a}, D_{b}] = -D_{[a, b]}$.

For all $\bar{a} \in L/Z(L)$ and any $\bar{d} \in S$, we have
\[D_{[a, d]} = -[D_{a}, D_{d}] \in D_{S},\]
which implies that $\overline{[a, d]} \in S$, i.e., $[\bar{a}, \bar{d}] \in S$. Hence, $S$ is a non trivial ideal of $L/Z(L)$, contradicting with $L/Z(L)$ is simple.

Hence, $Der(J)$ is simple.

(2). Suppose that $Der(J)$ is simple. We prove that $L/Z(L)$ is simple by using reduction to absurdity.

Otherwise, there exists a non trivial ideal of $L/Z(L)$, denoted by $S$. Let $D_{S} = \{D_{d} \mid \bar{d} \in S\}$. It's obvious that $D_{S}$ is a subspace of $Der(J)$.

Since $S \neq \{\bar{0}\}$, there exists a nonzero element $\bar{d_{0}} \in S$, i.e., $d_{0} \notin Z(L)$. Then there exists $0 \neq a \in L$ such that $[a, d_{0}] \neq 0$, i.e., $D_{d_{0}}(a) \neq 0$. Hence, $D_{d_{0}} \neq 0$. Therefore, $D_{S} \neq \{0\}$.

Assume that $D_{d_{1}} = D_{d_{2}}$ where $d_{1}, d_{2} \in L$, then for any $a \in J$
\[[a, d_{1}] = D_{d_{1}}(a) = D_{d_{2}}(a) = [a, d_{2}],\]
which implies that $d_{1} - d_{2} \in Z(L)$, i.e., $\bar{d_{1}} = \bar{d_{2}}$. Hence, $D_{d_{1}} = D_{d_{2}}$ if and only if $\bar{d_{1}} = \bar{d_{2}}$.

Since $S \neq L/Z(L)$, there exists a nonzero element $\bar{d_{1}} \in L/Z(L)$ and $\bar{d_{1}} \notin S$. That is to say $D_{d_{1}} \in Der(J)$ and $D_{d_{1}} \notin D_{S}$. Hence, $D_{S} \neq Der(J)$.

For all $D \in Der(J)$, according to Lemma \ref{le:2.3}(respectively, Lemma \ref{le:2.4}), there exists an element(respectively, a $P$-skew element) $d$ in $\mathscr{A}$ such that $D = D_{d}$. For all $\bar{d_{1}} \in S$, we have
\[[D, D_{d_{1}}] = [D_{d}, D_{d_{1}}] = -D_{[d, d_{1}]},\]
we have $\overline{[d, d_{1}]} = [\bar{d}, \bar{d_{1}}] \in S$ since $S$ is an ideal of $L/Z(L)$, hence $[D, D_{d_{1}}] \in D_{S}$.

Hence, we have $[Der(J), D_{S}] \subseteq D_{S}$, i.e., $D_{S}$ is a non trivial ideal of $Der(J)$, contradicting with $Der(J)$ is simple.

Therefore, $L/Z(L)$ is simple.
\end{proof}
\begin{lem}\label{le:2.7}
Let $J$ be a simple Jordan algebra of type $C$, i.e., $J$ has a basis $\{1, u_{1}, \cdots, u_{n}\}$ such that $1$ is the unit and
\[u^{2}_{i} = \alpha_{i}1\; where\; \alpha_{i} \neq 0, u_{i} \circ u_{j} = 0 \;if \;i \neq j.\]

Then $Der(J)$ is isomorphic to $M$ where $M$ is the Lie algebra generated by all $n \times n$-matrices $(v_{ij})_{n \times n}$ satisfying $\alpha_{i}v_{ji} + \alpha_{j}v_{ij} = 0$.
\end{lem}
\begin{proof}
First we'll show that $M$ is closed under the usual Lie bracket. It's obvious that $\{E_{ij} - \frac{\alpha_{j}}{\alpha_{i}}E_{ji}|1 \leq i < j \leq n\}$ is a basis of $M$. We have
$$[E_{ij} - \frac{\alpha_{j}}{\alpha_{i}}E_{ji}, E_{kl} - \frac{\alpha_{l}}{\alpha_{k}}E_{lk}] =
\left\{
\begin{aligned}
\frac{\alpha_{l}}{\alpha_{i}}(E_{lj} - \frac{\alpha_{j}}{\alpha_{l}}E_{jl}), (i = k, j \neq l)\\
\frac{\alpha_{j}}{\alpha_{i}}(E_{ki} - \frac{\alpha_{i}}{\alpha_{k}}E_{ik}), (i \neq k, j = l)\\
E_{il} - \frac{\alpha_{l}}{\alpha_{i}}E_{li},(j = k, i \neq l)\\
-(E_{kj} - \frac{\alpha_{j}}{\alpha_{k}}E_{jk}),(j \neq k, i = l)
\end{aligned}
\right.
$$
others are all zero.

Hence, $M$ is a Lie algebra under the usual Lie bracket.

Now we will show that $Der(J)$ is isomorphic to $M$. Define $D : J \rightarrow J$ to be a linear map by
\begin{equation}
\left\{
\begin{aligned}
D(1) = a_{0}1 + \sum^{n}_{k = 1}b_{0k}u_{k},\\
D(u_{i}) = a_{i}1 + \sum^{n}_{k = 1}b_{ik}u_{k}.
\end{aligned}
\right.
\end{equation}

Suppose that $D$ is a derivation on $J$, i.e., $D(x \circ y) = D(x) \circ y + x \circ D(y),\quad\forall x, y \in J$. It's only need to verify $D$ satisfies the above equation on base elements.

Take $x = 1$, $y = u_{i}(1 \leq i \leq n)$, we have
\[D(1 \circ u_{i}) = D(u_{i}) = a_{i}1 + \sum^{n}_{k = 1}b_{ik}u_{k},\]
\begin{align*}
&D(1) \circ u_{i} + 1 \circ D(u_{i}) = (a_{0}1 + \sum^{n}_{k = 1}b_{0k}u_{k}) \circ u_{i} + 1 \circ (a_{i}1 + \sum^{n}_{k = 1}b_{ik}u_{k})\\
&= (a_{0} + b_{ii})u_{i} + (\alpha_{i}b_{0i} + a_{i})1 + \sum^{n}_{k = 1, k \neq i}b_{ik}u_{k},
\end{align*}
comparing the coefficients on both sides, we have
\begin{equation}
\left\{
\begin{aligned}
a_{i} = \alpha_{i}b_{0i} + a_{i},\\
b_{ii} = a_{0} + b_{ii},
\end{aligned}
\right.
\end{equation}
note that $\alpha_{i} \neq 0$, we have
\begin{equation}
\left\{
\begin{aligned}
b_{0i} = 0,\\
a_{0} = 0.
\end{aligned}
\right.
\end{equation}

Similarly, take $x = u_{i}, y = u_{j}, (1 \leq i, j \leq n, i \neq j)$, we have
\begin{equation}
\left\{
\begin{aligned}
\alpha_{i}b_{ji} + \alpha_{j}b_{ij} = 0,\\
a_{i} = 0.
\end{aligned}
\right.
\end{equation}
Take $x = y = u_{i}, 1 \leq i \leq n$, we have
\begin{equation}
\left\{
\begin{aligned}
b_{ii} = 0,\\
a_{i} = 0.
\end{aligned}
\right.
\end{equation}

Hence, while $b_{ji} \neq 0(i \neq j)$ satisfying $\alpha_{i}b_{ji} + \alpha_{j}b_{ij} = 0(i \neq j)$ and others are all zero, $D$ is a derivation on $J$. We denote the matrix of $D$ under the basis $\{1, u_{1}, \cdots, u_{n}\}$ by
$$N_{D} =
\begin{pmatrix}
0 & 0 & 0 & 0 & \cdots & 0 & 0\\
0 & 0 & b_{12} & b_{13} & \cdots & b_{1, n - 1} & b_{1n}\\
0 & b_{21} & 0 & b_{23} & \cdots & b_{2, n - 1} & b_{2n}\\
\vdots & \vdots & \vdots & \vdots & \vdots & \vdots & \vdots\\
0 & b_{n1} & b_{n2} & b_{n3} & \cdots & b_{n, n - 1} & 0
\end{pmatrix}
,$$
i.e., $D(1, u_{1}, \cdots, u_{n}) = N_{D}(1, u_{1}, \cdots, u_{n})^{T}$.

We set $N$ be the Lie algebra generated the above matrixes. It's obvious that there is a one-to-one correspondence between $Der(J)$ and $N$ by mapping every $D \in Der(J)$ to $N_{D}$, i.e., $Der(J)$ is isomorphic to $N$.

Define $f : N \rightarrow M$ to be a linear map by $f(N_{D}) = M_{D}$ where
$$M_{D} =
\begin{pmatrix}
0 & b_{12} & b_{13} & \cdots & b_{1, n - 1} & b_{1n}\\
b_{21} & 0 & b_{23} & \cdots & b_{2, n - 1} & b_{2n}\\
\vdots & \vdots & \vdots & \vdots & \vdots & \vdots\\
b_{n1} & b_{n2} & b_{n3} & \cdots & b_{n, n - 1} & 0
\end{pmatrix}
.$$

It's easy to verify that $f$ is an isomorphism between $N$ and $M$. Hence $Der(J)$ is isomorphic to $M$ where $M$ is the Lie algebra generated by all $n \times n$-matrices $(v_{ij})_{n \times n}$ satisfying $\alpha_{i}v_{ji} + \alpha_{j}v_{ij} = 0$.
\end{proof}
\begin{thm}\label{thm:2.8}
Let $J$ be a simple Jordan algebra of type $C$. Then $Der(J)$ is simple if and only if $dim(J) \neq 5$.
\end{thm}
\begin{proof}
We'll prove our conclusion by contradiction. According to Lemma \ref{le:2.7}, $Der(J)$ has a basis $\{E_{ij} - \frac{\alpha_{j}}{\alpha_{i}}E_{ji}|1 \leq i < j \leq n\}$. Suppose $I$ is a non trivial ideal of $Der(J)$. Take $0 \neq a \in I$ and suppose $a = \sum^{n}_{i, j = 1, i < j}a_{ij}(E_{ij} - \frac{\alpha_{j}}{\alpha_{i}}E_{ji})$.

$\bf Step 1$. We'll show that there exists $1 \leq i_{0} < j_{0} \leq n$ such that $E_{i_{0}j_{0}} - \frac{\alpha_{j_{0}}}{\alpha_{i_{0}}}E_{j_{0}i_{0}} \in I$.

When $n = 2$, the dimension of $Der(J)$ is $1$. The conclusion is clearly true.

When $n = 3$, the dimension of $Der(J)$ is $3$. $Der(J) = \{E_{12} - \frac{\alpha_{2}}{\alpha_{1}}E_{21}, E_{13} - \frac{\alpha_{3}}{\alpha_{1}}E_{31}, E_{23} - \frac{\alpha_{3}}{\alpha_{2}}E_{32}\}$.

Since $a \neq 0$, there exists at least one non zero element in $S = \{a_{ij}|1 \leq i < j \leq 3\}$.

If there only exists one non zero element in $S$, it's obvious that there exists $1 \leq i_{0} < j_{0} \leq 3$ such that $E_{i_{0}j_{0}} - \frac{\alpha_{j_{0}}}{\alpha_{i_{0}}}E_{j_{0}i_{0}} \in I$.

If there exist two non zero elements in $S$,  we might as well set $a_{12}, a_{13} \neq 0$, i.e., $a = a_{12}(E_{12} - \frac{\alpha_{2}}{\alpha_{1}}E_{21}) + a_{13}(E_{13} - \frac{\alpha_{3}}{\alpha_{1}}E_{31})$.
\[[a, E_{13} - \frac{\alpha_{3}}{\alpha_{1}}E_{31}] = -a_{12}\frac{\alpha_{2}}{\alpha_{1}}(E_{23} - \frac{\alpha_{3}}{\alpha_{2}}E_{32}).\]
Hence, there exists $1 \leq i_{0} < j_{0} \leq 3$ such that $E_{i_{0}j_{0}} - \frac{\alpha_{j_{0}}}{\alpha_{i_{0}}}E_{j_{0}i_{0}} \in I$.

If there exist three non zero elements in $S$, i.e, $a = a_{12}(E_{12} - \frac{\alpha_{2}}{\alpha_{1}}E_{21}) + a_{13}(E_{13} - \frac{\alpha_{3}}{\alpha_{1}}E_{31}) + a_{23}(E_{23} - \frac{\alpha_{3}}{\alpha_{2}}E_{32})$.
\[[a, E_{23} - \frac{\alpha_{3}}{\alpha_{2}}E_{32}] = a_{12}(E_{13} - \frac{\alpha_{3}}{\alpha_{1}}E_{31}) - a_{13}\frac{\alpha_{3}}{\alpha_{2}}(E_{12} - \frac{\alpha_{2}}{\alpha_{1}}E_{21}),\]
\[[[a, E_{23} - \frac{\alpha_{3}}{\alpha_{2}}E_{32}], E_{12} - \frac{\alpha_{2}}{\alpha_{1}}E_{21}] =  a_{12}\frac{\alpha_{2}}{\alpha_{1}}(E_{23} - \frac{\alpha_{3}}{\alpha_{2}}E_{32}).\]
Hence, there exists $1 \leq i_{0} < j_{0} \leq 3$ such that $E_{i_{0}j_{0}} - \frac{\alpha_{j_{0}}}{\alpha_{i_{0}}}E_{j_{0}i_{0}} \in I$.

When $n \geq 5$, we set $S = \{a_{ij}|1 \leq i < j \leq n\}$. We prove the conclusion by induction on the number of nonzero elements in $S$, denoted by $k$.

Since $a \neq 0$, there exists at least one non zero element in $S = \{a_{ij}|1 \leq i < j \leq n\}$.

When $k = 1$, it's obvious that there exists $1 \leq i_{0} < j_{0} \leq n$ such that $E_{i_{0}j_{0}} - \frac{\alpha_{j_{0}}}{\alpha_{i_{0}}}E_{j_{0}i_{0}} \in I$.

When $k = 2$, we set $a_{i_{1}j_{1}}, a_{i_{2}j_{2}} \neq 0$, i.e., $a = a_{i_{1}j_{1}}(E_{i_{1}j_{1}} - \frac{\alpha_{j_{1}}}{\alpha_{i_{1}}}E_{j_{1}i_{1}}) + a_{i_{2}j_{2}}(E_{i_{2}j_{2}} - \frac{\alpha_{j_{2}}}{\alpha_{i_{2}}}E_{j_{2}i_{2}})$.

When $i_{1} = i_{2}, j_{1} \neq j_{2}$, we have $i_{2} \neq j_{1}$. Take $l \neq i_{1}, i_{2}, j_{1}, j_{2}$, then
\[[a, E_{j_{1}l} - \frac{\alpha_{l}}{\alpha_{j_{1}}}E_{lj_{1}}] = a_{i_{1}j_{1}}(E_{i_{1}l} - \frac{\alpha_{l}}{\alpha_{i_{1}}}E_{li_{1}}) \in I.\]
When $i_{1} \neq i_{2}, j_{1} = j_{2}$, we have $i_{1} \neq j_{2}$. Take $l \neq i_{1}, i_{2}, j_{1}, j_{2}$, then
\[[a, E_{li_{1}} - \frac{\alpha_{i_{1}}}{\alpha_{l}}E_{i_{1}l}] = -a_{i_{1}j_{1}}(E_{lj_{1}} - \frac{\alpha_{j_{1}}}{\alpha_{l}}E_{j_{1}l}) \in I.\]
When $i_{1} \neq i_{2}, j_{1} \neq j_{2}$, take $l \neq i_{1}, i_{2}, j_{1}, j_{2}$, then
\begin{enumerate}[(1)]
\item When $i_{2} \neq j_{1}$, $[a, E_{j_{1}l} - \frac{\alpha_{l}}{\alpha_{j_{1}}}E_{lj_{1}}] = a_{i_{1}j_{1}}(E_{i_{1}l} - \frac{\alpha_{l}}{\alpha_{i_{1}}}E_{li_{1}}) \in I$.\\
\item When $i_{2} = j_{1}$, $[a, E_{j_{1}l} - \frac{\alpha_{l}}{\alpha_{j_{1}}}E_{lj_{1}}] = a_{i_{1}j_{1}}(E_{i_{1}l} - \frac{\alpha_{l}}{\alpha_{i_{1}}}E_{li_{1}}) + a_{i_{2}j_{2}}\frac{\alpha_{l}}{\alpha_{i_{2}}}(E_{lj_{2}} - \frac{\alpha_{j_{2}}}{\alpha_{l}}E_{j_{2}l}) \in I$;

    $[[a, E_{j_{1}l} - \frac{\alpha_{l}}{\alpha_{j_{1}}}E_{lj_{1}}], E_{lj_{2}} - \frac{\alpha_{j_{2}}}{\alpha_{l}}E_{j_{2}l}] = a_{i_{1}j_{1}}(E_{i_{1}j_{2}} - \frac{\alpha_{j_{2}}}{\alpha_{i_{1}}}E_{j_{2}i_{1}}) \in I$.
\end{enumerate}
Hence, there exists $1 \leq i_{0} < j_{0} \leq n$ such that $E_{i_{0}j_{0}} - \frac{\alpha_{j_{0}}}{\alpha_{i_{0}}}E_{j_{0}i_{0}} \in I$.

Assume that the conclusion holds when there are $k(k < \frac{n(n - 1)}{2})$ nonzero elements in $S$. While there are $k + 1$ nonzero elements in $S$, suppose that $a = \sum^{k + 1}_{l = 1}a_{i_{l}j_{l}}(E_{i_{l}j_{l}} - \frac{\alpha_{j_{l}}}{\alpha_{i_{l}}}E_{j_{l}i_{l}})$.

Write $b = \sum^{k}_{l = 1}a_{i_{l}j_{l}}(E_{i_{l}j_{l}} - \frac{\alpha_{j_{l}}}{\alpha_{i_{l}}}E_{j_{l}i_{l}})$. Then $a = b + a_{i_{k + 1}j_{k + 1}}(E_{i_{k + 1}j_{k + 1}} - \frac{\alpha_{j_{k + 1}}}{\alpha_{i_{k + 1}}}E_{j_{k + 1}i_{k + 1}})$.

According to the assumption, $b$ becomes $p(E_{i^{'}j^{'}} - \frac{\alpha_{j^{'}}}{\alpha_{i^{'}}}E_{j^{'}i^{'}})$ after finite times of Lie brackets where $p$ is a nonzero coefficient. Meanwhile, $a_{i_{k + 1}j_{k + 1}}(E_{i_{k + 1}j_{k + 1}} - \frac{\alpha_{j_{k + 1}}}{\alpha_{i_{k + 1}}}E_{j_{k + 1}i_{k + 1}})$ vanishes or becomes $q(E_{i^{''}j^{''}} - \frac{\alpha_{j^{''}}}{\alpha_{i^{''}}}E_{j^{''}i^{''}})$ after the same finite times Lie brackets where $q$ is a nonzero coefficient. Hence, $a$ becomes $p(E_{i^{'}j^{'}} - \frac{\alpha_{j^{'}}}{\alpha_{i^{'}}}E_{j^{'}i^{'}})$ or $p(E_{i^{'}j^{'}} - \frac{\alpha_{j^{'}}}{\alpha_{i^{'}}}E_{j^{'}i^{'}}) + q(E_{i^{''}j^{''}} - \frac{\alpha_{j^{''}}}{\alpha_{i^{''}}}E_{j^{''}i^{''}})$. We have $p(E_{i^{'}j^{'}} - \frac{\alpha_{j^{'}}}{\alpha_{i^{'}}}E_{j^{'}i^{'}}) \in I$ or $p(E_{i^{'}j^{'}} - \frac{\alpha_{j^{'}}}{\alpha_{i^{'}}}E_{j^{'}i^{'}}) + q(E_{i^{''}j^{''}} - \frac{\alpha_{j^{''}}}{\alpha_{i^{''}}}E_{j^{''}i^{''}}) \in I$ since $I$ is an ideal of $Der(J)$. According to the case of $k = 2$, the conclusion holds.

Therefore, there exists $1 \leq i_{0} < j_{0} \leq n$ such that $E_{i_{0}j_{0}} - \frac{\alpha_{j_{0}}}{\alpha_{i_{0}}}E_{j_{0}i_{0}} \in I$.

$\bf Step 2$. We'll show that $E_{n - 1, n} - \frac{\alpha_{n}}{\alpha_{n - 1}}E_{n, n - 1} \in I$.

We have
\[[E_{i_{0}j_{0}} - \frac{\alpha_{j_{0}}}{\alpha_{i_{0}}}E_{j_{0}i_{0}}, E_{j_{0}n} - \frac{\alpha_{n}}{\alpha_{j_{0}}}E_{nj_{0}}] = E_{i_{0}n} - \frac{\alpha_{n}}{\alpha_{i_{0}}}E_{ni_{0}} \in I,\]
\begin{enumerate}[(1)]
\item if $i_{0} = n - 1$, the conclusion is proved;\\

\item if $i_{0} < n - 1$, then $[E_{i_{0}n} - \frac{\alpha_{n}}{\alpha_{i_{0}}}E_{ni_{0}}, E_{i_{0}n - 1} - \frac{\alpha_{n - 1}}{\alpha_{i_{0}}}E_{n - 1, i_{0}}] = \frac{\alpha_{n - 1}}{\alpha_{i_{0}}}(E_{n - 1, n} - \frac{\alpha_{n}}{\alpha_{n - 1}}E_{n, n - 1}) \in I$.
\end{enumerate}

$\bf Step 3$. We'll show that $I = Der(J)$.

Since
\[[E_{n - 1, n} - \frac{\alpha_{n}}{\alpha_{n - 1}}E_{n, n - 1}, E_{k, n - 1} - \frac{\alpha_{n - 1}}{\alpha_{k}}E_{n - 1, k}] = -(E_{kn} - \frac{\alpha_{n}}{\alpha_{k}}E_{nk}) \in I,(1 \leq k < n - 1),\]
we have $E_{1n} - \frac{\alpha_{n}}{\alpha_{1}}E_{n1}, \cdots, E_{n - 1, n} - \frac{\alpha_{n}}{\alpha_{n - 1}}E_{n, n - 1} \in I$.

Since
\[[E_{n - 1, n} - \frac{\alpha_{n}}{\alpha_{n - 1}}E_{n, n - 1}, E_{kn} - \frac{\alpha_{n}}{\alpha_{k}}E_{nk}] = \frac{\alpha_{n}}{\alpha_{n - 1}}(E_{k,n - 1} - \frac{\alpha_{n - 1}}{\alpha_{k}}E_{n - 1, k}) \in I,(1 \leq k < n - 1),\]
we have $E_{1, n - 1} - \frac{\alpha_{n - 1}}{\alpha_{1}}E_{n - 1, 1}, \cdots, E_{n - 2, n - 1} - \frac{\alpha_{n - 1}}{\alpha_{n - 2}}E_{n - 1, n - 2} \in I$.

Repeat the process above, we can get all $\{E_{ij} - \frac{\alpha_{j}}{\alpha_{i}}E_{ji}|1 \leq i < j \leq n\} \in I$, i.e., $Der(J) \subseteq I$, which implies that $Der(J) = I$, contradiction. Therefore, $Der(J)$ is simple.

When $n = 4$, then $Der(J)$ has a basis $\{E_{12} - \frac{\alpha_{2}}{\alpha_{1}}E_{21}, E_{13} - \frac{\alpha_{3}}{\alpha_{1}}E_{31}, E_{14} - \frac{\alpha_{4}}{\alpha_{1}}E_{41}, E_{23} - \frac{\alpha_{3}}{\alpha_{2}}E_{32}, E_{24} - \frac{\alpha_{4}}{\alpha_{2}}E_{42}, E_{34} - \frac{\alpha_{4}}{\alpha_{3}}E_{43}\}$. Set $I = \{ -\sqrt{\frac{\alpha_{4}\alpha_{1}}{\alpha_{2}\alpha_{3}}}(E_{12} - \frac{\alpha_{2}}{\alpha_{1}}E_{21}) + (E_{34} - \frac{\alpha_{4}}{\alpha_{3}}E_{43}), \sqrt{\frac{\alpha_{4}\alpha_{1}}{\alpha_{2}\alpha_{3}}}(E_{13} - \frac{\alpha_{3}}{\alpha_{1}}E_{31}) + (E_{24} - \frac{\alpha_{4}}{\alpha_{2}}E_{42}), -\sqrt{\frac{\alpha_{3}\alpha_{1}}{\alpha_{2}\alpha_{4}}}(E_{14} - \frac{\alpha_{4}}{\alpha_{1}}E_{41}) + (E_{23} - \frac{\alpha_{3}}{\alpha_{2}}E_{32})\}$. It's easy to verify that $I$ is an ideal of $Der(J)$. Hence, When $dim(J) = 5$, $Der(J)$ is not simple.
\end{proof}
\begin{lem}\cite{CS1}\label{le:2.9}
Let $\rm{F}$ be an algebraically closed field of characteristic $0$. Then the derivation algebra of the Jordan algebra of type $D$ is the Lie algebra $F_{4}$.
\end{lem}
\begin{thm}\label{thm:2.10}
The derivation algebra of the Jordan algebra of type $D$ over an algebraically closed field of characteristic $0$ is simple.
\end{thm}
\begin{proof}
According to Lemma \ref{le:2.9}, the derivation algebra of the Jordan algebra of type $D$ over an algebraically closed field of characteristic $0$ is the Lie algebra $F_{4}$. Note that $F_{4}$ is simple, the proof is completed.
\end{proof}

\section{Triple derivations of Jordan algebras}\label{se:3}

\begin{defn}
Let $J$ be a Jordan algebra over a field $\rm{F}$. A linear map $D : J \rightarrow J$ is called a triple derivation on $J$ if it satisfies
\[D((x \circ y) \circ z) = (D(x) \circ y) \circ z + (x \circ D(y)) \circ z + (x \circ y) \circ D(z),\quad\forall x, y, z \in J.\]
\end{defn}

It's easy to verify that derivations are all triple derivations for Jordan algebras. But the reverse is not always true. One can see the following example.
\begin{ex}\label{ex:3.2}
Let $J$ be a Jordan algebra with a basis $\{e_{1}, e_{2}\}$. And the multiplication table is
\[e_{1} \circ e_{1} = e_{2} \circ e_{2} = e_{1} + e_{2}, e_{1} \circ e_{2} = e_{2} \circ e_{1} = -e_{1} - e_{2}.\]
Then $J$ is a nilpotent Jordan algebra since $J^{3} = (J \circ J) \circ J = 0$. Therefore any linear map on $J$ is a triple derivation on $J$. Take $D_{0} : J \rightarrow J$ to be a linear map with
\[D_{0}(e_{1}) = D_{0}(e_{2}) = e_{1}.\]
Then we have
\[D_{0}(e_{1} \circ e_{1}) = D_{0}(e_{1} + e_{2}) = D_{0}(e_{1}) + D_{0}(e_{2}) = 2e_{1},\]
\[D_{0}(e_{1}) \circ e_{1} + e_{1} \circ D_{0}(e_{1}) = 2e_{1} \circ e_{1} = 2(e_{1} + e_{2}),\]
since
\[D_{0}(e_{1} \circ e_{1}) \neq D_{0}(e_{1}) \circ e_{1} + e_{1} \circ D_{0}(e_{1}),\]
$D_{0}$ is not a derivation on $J$. Hence, $TDer(J) \neq Der(J)$.
\end{ex}
\begin{lem}
For any Jordan algebra $J$, $TDer(J)$ is closed under the usual Lie bracket.
\end{lem}
\begin{proof}
$\forall D_{1}, D_{2} \in TDer(J)$, $\forall x, y, z \in J$,
\begin{align*}
&[D_{1}, D_{2}]((x \circ y) \circ z) = (D_{1}D_{2} - D_{2}D_{1})((x \circ y) \circ z)\\
&= D_{1}D_{2}((x \circ y) \circ z) - D_{2}D_{1}((x \circ y) \circ z)\\
&= D_{1}((D_{2}(x) \circ y) \circ z + (x \circ D_{2}(y)) \circ z + (x \circ y) \circ D_{2}(z)) - D_{2}((D_{1}(x) \circ y) \circ z +\\
&(x \circ D_{1}(y)) \circ z + (x \circ y) \circ D_{1}(z))\\
&= (D_{1}D_{2}(x) \circ y) \circ z + (D_{2}(x) \circ D_{1}(y)) \circ z + (D_{2}(x) \circ y) \circ D_{1}(z)\\
&+ (D_{1}(x) \circ D_{2}(y)) \circ z + (x \circ D_{1}D_{2}(y)) \circ z + (x \circ D_{2}(y)) \circ D_{1}(z)\\
&+ (D_{1}(x) \circ y) \circ D_{2}(z) + (x \circ D_{1}(y)) \circ D_{2}(z) + (x \circ y) \circ D_{1}D_{2}(z)\\
&- (D_{2}D_{1}(x) \circ y) \circ z - (D_{1}(x) \circ D_{2}(y)) \circ z - (D_{1}(x) \circ y) \circ D_{2}(z)\\
&- (D_{2}(x) \circ D_{1}(y)) \circ z - (x \circ D_{2}D_{1}(y)) \circ z - (x \circ D_{1}(y)) \circ D_{2}(z)\\
&- (D_{2}(x) \circ y) \circ D_{1}(z) - (x \circ D_{2}(y)) \circ D_{1}(z) - (x \circ y) \circ D_{2}D_{1}(z)\\
&= ([D_{1}, D_{2}](x) \circ y) \circ z + (x \circ [D_{1}, D_{2}](y)) \circ z + (x \circ y) \circ [D_{1}, D_{2}](z).
\end{align*}

Hence, $[D_{1}, D_{2}] \in TDer(J)$. The lemma is proved.
\end{proof}
\begin{thm}\label{thm:3.4}
Suppose that $J$ is a Jordan algebra with unit $1$ and $D : J \rightarrow J$ is a triple derivation. Then $D$ is a derivation on $J$ in the case the characteristic of the field $\rm{F}$ is not $2$. Moreover, we have $TDer(J) = Der(J)$.
\end{thm}
\begin{proof}
$\forall D \in TDer(J)$, we have
\[D(1) = D((1 \circ 1) \circ 1) = (D(1) \circ 1) \circ 1 + (1 \circ D(1)) \circ 1 + (1 \circ 1) \circ D(1) = 3D(1).\]

Note that $char \rm{F} \neq 2$, we have $D(1) = 0$.

For all $x, y \in J$,
\begin{align*}
&D(x \circ y) = D((x \circ y) \circ 1) = (D(x) \circ y) \circ 1 + (x \circ D(y)) \circ 1 + (x \circ y) \circ D(1)\\
&= D(x) \circ y + x \circ D(y),
\end{align*}
which implies that $D \in Der(J)$.

Hence, we have $TDer(J) \subseteq Der(J)$. Therefore, we have $TDer(J) = Der(J)$.
\end{proof}
\begin{lem}\cite{A1}\label{le:3.5}
Let $J$ be a semi-simple Jordan algebra over a field $\rm{F}$ of characteristic $0$. Then $J$ has the unique unit.
\end{lem}
\begin{lem}\cite{N1}\label{le:3.6}
Every derivation of a semi-simple Jordan algebra with a finite basis over a field of characteristic $0$ is inner.
\end{lem}
According to Lemma \ref{le:3.5} and \ref{le:3.6}, we get a corollary with respect to Theorem \ref{thm:3.4}.
\begin{cor}\label{cor:3.7}
Suppose that $J$ is a semi-simple Jordan algebra with a finite basis over a field characteristic $0$. Then $TDer(J) = Der(J) = Inn(J)$.
\end{cor}

\section{Triple derivations of the derivation algebras of Jordan algebras}\label{se:4}

\begin{defn}\cite{Z1}
A linear map $D: L \rightarrow L$ where $L$ is a Lie algebra, is called a triple derivation on $L$ if it satisfies
\[D([[x, y], z]) = [[D(x), y], z] + [[x, D(y)], z] + [[x, y], D(z)],\quad\forall x, y, z \in L.\]
\end{defn}
\begin{lem}\cite{Z1}\label{le:4.2}
Let $L$ be a Lie algebra over commutative ring $\rm{R}$. If $\frac{1}{2} \in \rm{R}$, $L$ is perfect and has zero center, then we have that:
\begin{enumerate}[(1)]
\item $TDer(L) = Der(L)$;
\item $TDer(Der(L)) = ad(Der(L))$.
\end{enumerate}
\end{lem}
\begin{lem}\cite{E1}\label{le:4.3}
Any derivation of a semi-simple Lie algebra $L$ over a field of characteristic of $0$ is inner.
\end{lem}

According to Lemma \ref{le:4.2} and \ref{le:4.3}, we have the following theorem.
\begin{thm}\label{thm:4.4}
Let $J$ be a simple Jordan algebra of type $A$ or $B$ where $L/Z(L)$ in Theorem \ref{thm:2.6} is simple or a simple Jordan algebra of type $C$ whose dimension isn't $5$ or a simple Jordan algebra of type $D$. Then $TDer(Der(J)) = Der(Der(J)) = ad(Der(J))$.
\end{thm}
\begin{lem}\cite{M2}\label{le:4.5}
Suppose that $L$ is a Lie algebra and has a decomposition $L = L_{1} \oplus L_{2}$, where $L_{1}$, $L_{2}$ are ideals of $L$. Then
\begin{enumerate}[(1)]
\item $C(L) = C(L_{1}) \oplus C(L_{2})$;
\item If\; $C(L) = \{0\}$, then $Der(L) = Der(L_{1}) \oplus Der(L_{2})$ and $ad(L) = ad(L_{1}) \oplus ad(L_{2})$.
\end{enumerate}
\end{lem}
\begin{lem}\label{le:4.6}
Suppose that $L$ is a Lie algebra and has a decomposition $L = L_{1} \oplus L_{2}$, where $L_{1}$, $L_{2}$ are ideals of $L$. If\; $C(L) = \{0\}$, then $TDer(L) = TDer(L_{1}) \oplus TDer(L_{2})$.
\end{lem}
\begin{proof}
$\bf{Step 1}$. We'll show that for $i = 1, 2$, $D(L_{i}) \subseteq L_{i},\;\forall D \in TDer(L)$.

Suppose that $x_{i} \in L_{i}(i = 1, 2)$. For any $z \in L$, we have
\[D([[x_{1}, x_{2}], z]) = [[D(x_{1}), x_{2}], z] + [[x_{1}, D(x_{2})], z] + [[x_{1}, x_{2}], D(z)],\]
which implies that
\[[[D(x_{1}), x_{2}], z] + [[x_{1}, D(x_{2})], z] = 0,\]
i.e,
\[[[D(x_{1}), x_{2}], z] = -[[x_{1}, D(x_{2})], z] \in L_{1} \cap L_{2} = 0.\]
Hence,
\[[D(x_{1}), x_{2}] \in C(L), [x_{1}, D(x_{2})] \in C(L).\]
Note that $C(L) = \{0\}$, we have
\[[D(x_{1}), x_{2}] = [x_{1}, D(x_{2})] = 0.\]
Suppose that $D(x_{1}) = u_{1} + u_{2}$ where $u_{i} \in L_{i}(i = 1, 2)$. Then we have
\[[u_{2}, x_{2}] = [u_{1} + u_{2}, x_{2}] = [D(x_{1}), x_{2}] = 0,\]
which implies that $u_{2} \in C(L_{2})$.
According to Lemma \ref{le:4.5}, $C(L_{i}) = \{0\}(i = 1, 2)$. Therefore $u_{2} = 0$, which is to say $D(L_{1}) \subseteq L_{1}$.

Similarly, we have $D(L_{2}) \subseteq L_{2}$.

$\bf{Step 2}$. we'll show that $TDer(L_{1}) \dotplus TDer(L_{2}) \subseteq TDer(L)$.

For any $D \in TDer(L_{1})$, we extend it to a linear map on $L$ as follow
\[D(x_{1} + x_{2}) = x_{1},\quad\forall x_{1} \in L_{1}, x_{2} \in L_{2}.\]

Then for any $x, y, z \in L$, suppose that $x = x_{1} + x_{2}$, $y = y_{1} + y_{2}$, $z = z_{1} + z_{2}$, where $x_{1}, y_{1}, z_{1} \in L_{1}$, $x_{2}, y_{2}, z_{2} \in L_{2}$, we have
\[D([[x, y], z]) = D([[x_{1} + x_{2}, y_{1} + y_{2}], z_{1} + z_{2}]) = D([[x_{1}, y_{1}], z_{1}] + [[x_{2}, y_{2}], z_{2}]) = D([[x_{1}, y_{1}], z_{1}]),\]
\begin{align*}
&[[D(x), y], z] + [[x, D(y)], z] + [[x, y], D(z)]\\
&= [[D(x_{1} + x_{2}), y_{1} + y_{2}], z_{1} + z_{2}] + [[x_{1} + x_{2}, D(y_{1} + y_{2})], z_{1} + z_{2}] + [[x_{1} + x_{2}, y_{1} + y_{2}],\\
&  D(z_{1} + z_{2})]\\
&= [[D(x_{1}), y_{1}], z_{1}] + [[x_{1}, D(y_{1})], z_{1}] + [[x_{1}, y_{1}], D(z_{1})].
\end{align*}
Since $D \in TDer(L_{1})$, we have
\[D([[x_{1}, y_{1}], z_{1}]) = [[D(x_{1}), y_{1}], z_{1}] + [[x_{1}, D(y_{1})], z_{1}] + [[x_{1}, y_{1}], D(z_{1})].\]
Hence,
\[D([[x, y], z]) = [[D(x), y], z] + [[x, D(y)], z] + [[x, y], D(z)],\]
which implies that $D \in TDer(L)$, i.e, $TDer(L_{1}) \subseteq TDer(L)$. Moreover, $D \in TDer(L_{1})$ if and only if $D(x_{2}) = 0,\; \forall x_{2} \in L_{2}$.

Similarly, we have $TDer(L_{2}) \subseteq TDer(L)$ and $D \in TDer(L_{2})$ if and only if $D(x_{1}) = 0,\; \forall x_{1} \in L_{1}$.

Then we have $TDer(L_{1}) + TDer(L_{2}) \subseteq TDer(L)$ and $TDer(L_{1}) \cap TDer(L_{2}) = \{0\}$. Hence, $TDer(L_{1}) \dotplus TDer(L_{2}) \subseteq TDer(L)$.

$\bf{Step 3}$. We'll prove that $TDer(L_{1}) \dotplus TDer(L_{2}) = TDer(L)$.

Suppose that $D \in TDer(L)$. Set $x = x_{1} + x_{2}, x_{i} \in L_{i}$. Define $D_{1}, D_{2}$ as follows
\begin{equation}
\left\{
\begin{aligned}
D_{1}(x_{1} + x_{2}) = D(x_{1}),\\
D_{2}(x_{1} + x_{2}) = D(x_{2}).
\end{aligned}
\right.
\end{equation}
Obviously, $D = D_{1} + D_{2}$.

For any $u_{1}, v_{1}, w_{1} \in L_{1}$,
\begin{align*}
&D_{1}([[u_{1}, v_{1}], w_{1}]) = D([[u_{1}, v_{1}], w_{1}]) = [[D(u_{1}), v_{1}], w_{1}] + [[u_{1}, D(v_{1})], w_{1}] + [[u_{1}, v_{1}], D(w_{1})]\\
&= [[D_{1}(u_{1}), v_{1}], w_{1}] + [[u_{1}, D_{1}(v_{1})], w_{1}] + [[u_{1}, v_{1}], D_{1}(w_{1})].
\end{align*}
Hence, $D_{1} \in TDer(L_{1})$. Similarly, $D_{2} \in TDer(L_{2})$.

Therefore, $TDer(L) = TDer(L_{1}) \dotplus TDer(L_{2})$.

$\bf{Step 4}$. We'll show that $TDer(L_{i}) \lhd TDer(L)$. Suppose that $D_{1} \in TDer(L_{1}), D \in TDer(L), x_{2} \in L_{2}$.
\[[D, D_{1}](x_{2}) = DD_{1}(x_{2}) - D_{1}D(x_{2}) = 0,\]
which implies that $[D, D_{1}] \in TDer(L_{1})$, i.e., $TDer(L_{1}) \lhd TDer(L)$. Similarly, $TDer(L_{2}) \lhd TDer(L)$.

Therefore, we have $TDer(L) = TDer(L_{1}) \oplus TDer(L_{2})$.
\end{proof}
\begin{thm}\label{thm:4.7}
Suppose that $J$ is a centerless semi-simple Jordan algebra over a field of characteristic $0$ and $J$ has the decomposition $J = \oplus^{s}_{i = 1}J_{i}$ where $J_{i}(1 \leq i \leq s)$ are simple Jordan algebras over a field of characteristic $0$ satisfying $Der(J_{i})$ is simple. Then $TDer(Der(J)) = Der(Der(J)) = ad(Der(J))$.
\end{thm}
\begin{proof}
According to Theorem \ref{thm:2.2}, we have $Der(J) = \oplus^{s}_{i = 1}Der(J_{i})$ and $Der(J_{i}) \triangleleft Der(J)$.

By Lemma \ref{le:4.5}, we have $C(Der(J)) = \oplus^{s}_{i = 1}C(Der(J_{i}))$. Since $Der(J_{i})$ is simple, we have $C(Der(J_{i})) = \{0\}$. Hence, $C(Der(J)) = \{0\}$.

Therefore, we have
\[TDer(Der(J)) = \oplus^{s}_{i = 1}TDer(Der(J_{i})),\]
\[Der(Der(J)) = \oplus^{s}_{i = 1}Der(Der(J_{i})),\]
\[ad(Der(J)) = \oplus^{s}_{i = 1}ad(Der(J_{i})),\]
via Lemma \ref{le:4.5} and Lemma \ref{le:4.6}.

According to Theorem \ref{thm:4.4}, we have
\[TDer(Der(J_{i})) = Der(Der(J_{i})) = ad(Der(J_{i})),\quad\forall 1 \leq i \leq s.\]
Hence, we have
\[TDer(Der(J)) = Der(Der(J)) = ad(Der(J)).\]
\end{proof}
\begin{re}
Let $J$ be a simple Jordan algebra of type $C$. Obviously, $J$ is semi-simple. But $C(J) \neq \{0\}$ since $1 \in C(J)$.
\end{re}

\end{document}